\documentclass[9pt,twocolumn,twoside]{IEEEtran}
\usepackage{enumerate}
\usepackage{indentfirst}
\usepackage{graphicx} 
\usepackage[dvips]{epsfig}
\usepackage{tabularx}
\usepackage{graphicx} 
\usepackage[dvips]{epsfig} 
\usepackage{wrapfig} 
\usepackage{epstopdf} 
\usepackage{wrapfig} 
\usepackage{subfigure}
\usepackage{epstopdf}
\usepackage{amsmath}
\usepackage{amsfonts}
\usepackage{graphicx}
\usepackage{amssymb}
\usepackage{epsfig}
\usepackage{algorithm}
\usepackage{amsmath}
\usepackage{amsfonts}
\usepackage{graphicx}
\usepackage{amssymb}
\usepackage{epsfig}
\usepackage{color}

\newtheorem{condition**}{A*}
\newtheorem{condition***}{C*}
\newtheorem{condition*}{C}

\newtheorem{definition}{Definition}[section]
\newtheorem{theorem}{Theorem}[section]
\newtheorem{lemma}{Lemma}[section]

\newtheorem{assumption}{Assumption}[section]

\renewcommand{\vec}{\mathrm{vec}}

\begin{document}

\title{Stochastic Linear Quadratic Optimal Control Problem: A Reinforcement Learning Method}

\author{Na Li, Xun Li, Jing Peng and Zuo Quan Xu 
\thanks{Corresponding Author. N. Li is with School of Statistics,
Shandong University of Finance and Economics, Jinan, Shandong, 250014, China (e-mail: naibor@163.com). N. Li acknowledges the financial support by the NSFC (No.~11801317), and the Natural Science Foundation of Shandong Province (No.~ZR2019MA013) and the Colleges and Universities Youth Innovation Technology Program of Shandong Province (No.~2019KJI011).}
\thanks{ X. Li is with Department of Applied Mathematics, The Hong Kong Polytechnic University, Hong Kong, China (e-mail: li.xun@polyu.edu.hk). X. Li acknowledges the financial support by the Research Grants Council of Hong Kong under grant (15213218, 15215319 and 15216720), and the PolyU-SDU Joint Research Centre on Financial Mathematics.}
\thanks{J. Peng with Department of Applied Mathematics, The Hong Kong Polytechnic University, Hong Kong, China (e-mail:jing.peng@connect.polyu.hk). }
\thanks{ Z. Xu is with Department of Applied Mathematics, The Hong Kong Polytechnic University, Hong Kong, China (e-mail: maxu@polyu.edu.hk). Z. Xu acknowledges the financial support by NSFC (No.~11971409), and Hong Kong GRF (No.~15202817), the PolyU-SDU Joint Research Center on Financial Mathematics and the CAS AMSS-POLYU Joint Laboratory of Applied Mathematics, The Hong Kong Polytechnic University.}
}

\maketitle

\begin{abstract}
This paper applies a reinforcement learning (RL) method to solve infinite horizon continuous-time stochastic linear quadratic problems, where drift and diffusion terms in the dynamics may depend on both the state and control. Based on Bellman's dynamic programming principle, an online RL algorithm is presented to attain the optimal control with just partial system information. This algorithm directly computes the optimal control rather than estimating the system coefficients and solving the related Riccati equation. It just requires local trajectory information, greatly simplifying the calculation processing. Two numerical examples are carried out to shed light on our theoretical findings. 
\end{abstract}


%
\IEEEpeerreviewmaketitle

\section{Introduction}
Reinforcement learning (RL) is a hot topic in machine learning research, with roots in animal learning and early learning control work. 
Unlike other 
machine learning techniques such as supervised learning and unsupervised learning,  RL method focuses on optimizing the reward without explicitly exploiting the hidden structure. 
Two key features distinguish this approach: trial-and-error search and delayed rewards. 
One discovers the best strategy through trial and error, and his actions affect not only the immediate reward but also all later rewards. 
In this approach, the controller must first {\it exploit} his experience to give the control and then, based on the reward, {\it explore} new strategies for the future. The most significant challenge is the trade-off between exploitation and exploration. Please see \cite{Mendel-McLaren-1970, Sutton-Barto-2018,Wang-Zariphopoulou-Zhou-2019} for details.  

On the other hand, optimal control, along with regulation and tracking problems, is among the most important research topics in control theory (\cite{Berkovitz-1974,Yong-Zhou-1999}). When the appropriate model is not available, indirect and direct adaptive control techniques are utilized to provide the best control. The indirect method seeks to discover the system's structure and then derives the optimal control using the discovered system's information. By contrast, the direct method does not identify the structure of the system; instead, it adjusts the control directly to make the error between the plant output and desired output tend to zero asymptotically (see, Narendra and Valavani \cite{Narendra-Valavani-1978}). 

According to Sutton {\em et al.} \cite{Sutton-Barto-Williams-1991}, RL method may be seen as a direct approach to optimal control problems, as it computes the optimal controls directly without the need for the structure of the system. The significance of RL method is that it provides a new adaptive structure, which successively reinforces the reward function such that the adaptive controller converges to the optimal control. In comparison, indirect adaptive techniques must first estimate the system's structure before determining the control, which is intrinsically complicated.

Linear quadratic (LQ) problem is an important class of optimal control problems in both theory and practice, for it may reasonably simulate many nonlinear problems. This paper proposes an RL algorithm to solve stochastic LQ (SLQ) optimal control problems.

\subsection{{Related Work}} 
For deterministic optimal control problems, RL techniques have been extensively explored under both discrete-time and continuous-time frameworks. For instance, 
Bradtke {\em et al.} \cite{Bradtke-Ydestie-Barto-1994} gave a Q-learning policy iteration for a discrete-time LQ problem by the so-called Q-function (Watkins\cite{Watkins-1989}, Werbos \cite{Werbos-1989}). 
Q-learning is a widely used RL technique. For its recent applications to discrete-time models, we refer to Rizvi and Lin \cite{Rizvi-Lin-2019}, Luo {\em et al.} \cite{Luo-Liu-Wu-2018}, Kiumarsia {\em et al.} \cite{Kiumarsia-AlQaudi-Modaresa-Lewis-Levinee-2019}. Baird \cite{Baird-1994} firstly used RL approach to obtain the optimal control for a continuous-time discrete-state system. Murray {\em et al.} \cite{Murray-Cox-Lendaris-Saeks-2002}  proposed an iterative adaptive dynamic programming (ADP) scheme for nonlinear systems. 
Recently, a number of new RL methods are developed for optimal control problems in continuous-time cases { (e.g., \cite{Vrabie-Pastravanu-Abu-Khalaf-Lewis-2009, Jiang-Jiang-2015, Modares-Lewis-Jiang-2016, Yaghmaie-Braun-2019, Bian-Jiang-2019, Lee-Sutton-2021})}. 
Vrabie {\em et al.} \cite{Vrabie-Pastravanu-Abu-Khalaf-Lewis-2009} { proposed} a new policy iteration technique for continuous-time linear systems under partial information. Jiang and Jiang \cite{Jiang-Jiang-2015} studied a type of nonlinear polynomial system and proposed a novel ADP based on the Hamilton-Jacobi-Bellman (HJB) equation of a relaxed problem. 
Modares {\em et al.} \cite{Modares-Lewis-Jiang-2016} designed a model-free off-policy RL algorithm for a linear continuous-time system. Their method is also applicable to regulation and tracking problems. 
We refer to Kiumarsi {\em et al.} \cite{Kiumarsi-Vamvoudakis-Modares-Lewis-2018} and Chen {\em et al.} \cite{Chen-Qu-Tang-Low-Li-2021} for more related works.

An important approach to obtaining optimal control of SLQ problems on the infinite horizon is to solve the related stochastic algebraic Riccati equation (SARE). 
Ait Rami and Zhou \cite{Rami-Zhou-2000} tackled an indefinite SLQ control problem using analytical and computational approaches to treating the related SARE by semidefinite programming (SDP). Later, Huang {\em et al.} \cite{Huang-Li-Yong-2015} solved a kind of mean-field SLQ problem on the infinite horizon by SDP, which involves two coupled SAREs. Moreover, Sun and Yong \cite{Sun-Yong-2018} proved that the admissible control set is non-empty for every initial state, equivalent to the control system's stabilizability. Because SAREs are dependent on the coefficients in the dynamics and the cost functional, the algorithms based on SAREs must be implemented offline. 

Duncan {\em et al.} \cite{Duncan-Guo-Pasik-Duncan-1999} studied an SLQ problem for a linear diffusion system,
where coefficients of the drift term are not known, and the diffusion term is independent of the state and control. Their method is indirect: first adopt a weighted least squares algorithm to estimate the dynamics' coefficients and then give an adaptive LQ Gaussian control. Recently, academics have been increasingly interested in studying SLQ problems using RL techniques, even though a number of applications is highly restricted compared to deterministic problems. 
Wong and Lee \cite{Wong-Lee-2010} considered a discrete-time SLQ problem with white Gaussian signals by Q-learning. Their method is a direct approach. 
 Fazel {\em et al.} \cite{Fazel-Ge-Kakade-Mesbahi-2018} studied a time-homogeneous LQ regulator (LQR) problem with a random initial state. They found the optimal policy by a model-free local search method. 
The method provides the global convergence for the decent gradient methods and a higher convergence rate than the naive gradient method. 
Later, Mohammadi {\em et al.} \cite{Mohammadi-Soltanolkotabi-Jovanovic-2020} gave a random search method with two-point gradient estimates for continuous-time LQR problems. They improved the related works on the required function evaluations and simulation time. 
Wang {\em et al.} \cite{Wang-Zariphopoulou-Zhou-2019} applied RL technique to a non-linear stochastic continuous-time diffusion system based on the classical relaxed stochastic control { (see, for example, \cite{Fleming-Nisio-1984,Zhou-1992})}. The optimum trade-off between investigating the black box environment and using present information is accomplished. Following up them, Wang and Zhou \cite{Wang-Zhou-2019} developed a continuous-time mean-variance (MV) portfolio optimization problem. They derived a Gaussian feedback exploration policy. 
\subsection{{Motivation}}
We consider the model in this study primarily for two reasons, which will be stated separately in the following two paragraphs. 

The most notable advantage of LQ framework is that the optimal controls can usually be expressed in an explicit closed-form. To get the optimal control, one only needs to solve the related Riccati equation such as Ait Rami and Zhou \cite{Rami-Zhou-2000}. This approach requires all the information of the system. However, we sometimes only know the observation of the state process rather than all of the system's characteristics. Therefore, the SDP method may be impractical. 
As earlier mentioned, RL techniques may directly generate the optimal control using only the trajectory information. This motivates us to build a new RL algorithm to directly compute the optimal control rather than solving the Riccati problem. More precisely, the RL algorithm can learn what to do based on data along the trajectories; no complete system knowledge is required to implement our algorithm. 

As mentioned above, Duncan {\em et al.} \cite{Duncan-Guo-Pasik-Duncan-1999} studied an SLQ problem where the diffusion term is independent of the state and control. In financial and economic practice, however, decision makers' actions usually impact the trend of the system (drift term) and the uncertainty of the system (diffusion term). Therefore, it is necessary to consider the case where the diffusion term is affected by both the state and control. This motivates us to analyze a more comprehensive linear system where drift and diffusion terms depend on the state and control in this paper. The problem can also be viewed as the scenario where multiplicative noises are present in the state and control. Noises frequently have a multiplicative effect on various plant components; see a practical example in \cite{Ugrinovskii-1998}. 
Due to the presence of control in the diffusion term, the weighting matrix $R$ in the problem is allowed to be indefinite, which is a crucial instance in both theory and practice; see, for example, Chen {\em et al.} \cite{Chen-Li-Zhou-1998}, and \cite{Yong-Zhou-1999}. Although we explore the problem in this paper under the positive definite condition, the results established can naturally extend to the indefinite case.
\subsection{{Contribution}}
Inspired by the above observations, this paper develops an online RL algorithm to solve SLQ problems over infinite time horizon, primarily using stochastic Bellman dynamic programming (DP) rather than solving the related Riccati equation. The algorithm computes the best control based only on local trajectories rather than on the system's structure. In other words, our algorithm only focuses on getting the optimal control and does not intend to model the internal structure of the system. In practice, the controller only needs partial information of the system dynamics to get the optimal control by updating policy and improving the evaluator based on the online data of state trajectories. Our main contributions are stated as follows.
\begin{enumerate}[(i)]
\item The policy iteration is implemented along the trajectories online using only partial information of the system. To the best of our knowledge, this is the first time to study the SLQ problem for It\^o type continuous-time system with state and control in diffusion term by an RL method. As a byproduct, the solution of the Riccati equation is derived without solving the equation itself. 
\item Our algorithm only needs the local exploration on the time interval $[t,t+\Delta t]$,  with $t\geq0$ and $\Delta t>0$ being arbitrarily chosen. The stochastic DP allows us to adopt the optimality equation as the policy evaluation to reinforce the {\em target function} on a short interval $[t,t+\Delta t]$, rather than reinforce the cost functional on the entire infinite time horizon $[t, +\infty)$. This just requires local trajectory information, which greatly simplifies the calculation processing.
\item It is proved that, given a mean-square stabilizable control at initial, all the following up controls are stabilizable by our policy improvement. In contrast, Wang {\em et al.} \cite{Wang-Zariphopoulou-Zhou-2019} did not discuss the stabilizable issue. The convergence of the controls in our RL algorithm is also proved. 
\item Our RL algorithm is partially model-free, similar to Fazel {\em et al.} \cite{Fazel-Ge-Kakade-Mesbahi-2018} and Mohammadi {\em et al.} \cite{Mohammadi-Soltanolkotabi-Jovanovic-2020} that studied the problems with the random initial state in discrete-time and continuous-time, respectively. Differently, we study the It\^o type linear system with diffusion term and deterministic initial state. Moreover, the SLQ problem is also different from \cite{Wong-Lee-2010}, in which the system is only disturbed by white Gaussian signals. 
\end{enumerate}
The rest of this paper is organized as follows. Section \ref{LQ problem} presents an SLQ problem and gives an online RL algorithm to compute its optimal feedback control. We also discuss the properties such as stabilizable and convergence of the algorithm. 
We implement the algorithm and provide two numerical examples in Section \ref{Implementation}. 

{\it Notation.} Let $\mathbb N$ denote the set of positive integers. Let $n, m, L, K\in
\mathbb N$ be given. We denote by $\mathbb R^n$ the $n$-dimensional
Euclidean space with the norm $\|\cdot\|$. Let $\mathbb R^{n\times m}$ be the set of all
$n\times m$ real matrices. 
We denote by $A^{\top}$ the transpose of a vector or matrix $A$. Let $\mathcal S^n$ be the collection of all symmetric matrices in $\mathbb R^{n\times n}$. As usual, if a
matrix $A\in \mathcal S^n$ is positive semidefinite (resp. positive
definite), we write
$A\geq 0$ (resp. $>0$). All the positive
semidefinite (resp. positive
definite) matrices are collected by
$\mathcal S^n_+$ (resp. $\mathcal S^n_{++}$). 
If $A$, $B\in \mathcal S^n$, then we write $A\geq B$ (resp. $>$) if $A-B\geq 0$ (resp. $>0$). Denote $s, t\geq0$ as the time on finite horizon. 
Let $(\Omega, \mathcal F,
\mathbb P, \mathbb F)$ be a complete filtered probability space on
which a one-dimensional standard Brownian motion $W(\cdot)$ is
defined with $\mathbb F\equiv\{\mathcal F_t\}_{t\geq 0}$ being its
natural filtration augmented by all $\mathbb P$-null sets. We define the Hilbert space 
$L^2_{\mathbb F}(\mathbb R^n)$, which is the space of $\mathbb R^n$-valued $\mathbb F$-progressively measurable processes $\varphi(\cdot)$ with the finite norm $\|\varphi(\cdot)\|=\left[\mathbb E\int_{t}^\infty |\varphi(s)|^2 ds\right]^{\frac{1}{2}}<\infty$.  Furthermore, $\bf O$ denotes zero matrices with appropriate dimensions, and $\varnothing$ denotes the empty set.
\section{Online Algorithm for Stochastic LQ Optimal Control Problem}\label{LQ problem}


In this paper, we consider the following time-invariant stochastic linear dynamical control system 
\begin{equation}\label{system}
\left\{
\begin{aligned}
dX(s)&=\left[AX(s)+Bu(s)\right]ds\\&~~~\qquad+\left[CX(s)+Du(s)\right]dW(s),~~~s\geq t,\\
X(t)&=x\in\mathbb R^{n}, 
\end{aligned}
\right.
\end{equation}
where the coefficients $A$, $C\in\mathbb R^{n\times n}$ and $B$, $D\in\mathbb R^{n\times m}$ are constant matrices. The state process $X(\cdot)$ is an $n$-dimensional vector, the control $u$ is an $m$-dimensional vector, and $X(t)=x$ is the deterministic initial value. On the right hand of the system \eqref{system}, the first term is called the {\em drift term}, and the second term is called the {\em diffusion term}. Here, the
dimension of Brownian motion is set to be one just for simplicity of
notation, and the case of multi-dimensional Brownian motion can be
dealt with in the same way. We also denote this system by $[A,C; B,D]$ for simplicity.

\begin{definition}\label{Def} The system $[A,C; B,D]$ is called mean-square stabilizable (with respect to $x$) if there exists a constant matrix $K\in\mathbb R^{m\times n}$ such that the (unique) strong solution of 
\begin{equation}
\left\{
\begin{aligned}
dX(s)&=\left(A+BK\right)X(s)dt\\
&~~~\qquad+\left(C+DK\right)X(s)dW(s),~~~s\geq t,\\
X(t)&=x
\end{aligned}
\right.
\end{equation}
satisfies $\lim_{s\rightarrow\infty}\mathbb E[X(s)^\top X(s)]=0$. In this case, $K$ is called a stabilizer of the system $[A,C; B,D]$ and the feedback control $u(\cdot)=KX(\cdot)$ is called stabilizing. The set of all stabilizers is denoted by $\mathcal X=\mathcal X([A,C; B,D])$.
\end{definition}

The following assumption is used to avoid trivial cases. 
\begin{assumption}\label{ass1}
The system \eqref{system} is mean-square stabilizable, {\it i.e.},
\[
\mathcal X([A,C; B,D])\neq \varnothing. 
\]
\end{assumption}

The following result provides an equivalent condition for the existence of the stabilizers for the system \eqref{system}, please refer to Theorem 1 in \cite{Rami-Zhou-2000} or Lemma 2.2 in \cite{Sun-Yong-2018}.

\begin{lemma}\label{lemma-stabilizer} 
A matrix $K\in\mathbb R^{m\times n}$ is a stabilizer of the system $[A,C; B,D]$ if and only if 
there exists a matrix $P\in\mathcal S_{++}^n$ such that 
\begin{multline*}
(A+BK)^\top P+P(A+BK)\\
+(C+DK)^\top P(C+DK)<0.
\end{multline*}
In this case, for any $\Lambda\in\mathcal S^n$ (resp., $\mathcal S_{+}^n$, $\mathcal S_{++}^n$), the following Lyapunov equation
\begin{multline*}
(A+BK)^\top P+P(A+BK)\\
+(C+DK)^\top P(C+DK)+\Lambda=0
\end{multline*}
admits a unique solution $P\in\mathcal S^n$ (resp., $\mathcal S_{+}^n$, $\mathcal S_{++}^n$). 
\end{lemma}

This result shows that the set $\mathcal X([A,C; B,D])$ is, in fact, independent of the initial state $x$.
When the system $[A,C; B,D]$ is mean-square stabilizable, we define the corresponding set of admissible controls as
\begin{equation*}
\mathcal U_{ad}=\{u(\cdot)\in L_{\mathbb F}^2(\mathbb R^m): u(\cdot){\rm ~is~stabilizing}\}.
\end{equation*}
In this paper, we consider a quadratic cost functional given by 
\begin{multline}\label{cost}
J(t,x;u(\cdot))\\
=\mathbb E^{\mathcal F_t}\int_{t}^\infty\Big[X(s)^\top QX(s)+2u(s)^\top S X(s) +u(s)^\top Ru(s)\big]ds,
\end{multline}
where $Q$, $S$, and $R$ are given constant matrices of proper sizes. 

\noindent{\bf Problem (SLQ).} Given $t\geq 0$ and $x\in\mathbb R^n$, find a control $u^*(\cdot)\in\mathcal U_{ad}$ such that
\[
J(t,x;u^*(\cdot))=\inf_{u(\cdot)\in \mathcal U_{ad}}J(t,x;u(\cdot))\triangleq V(t, x),
\]
where $V(t,x)$ is called the value function of Problem (SLQ). 

Problem (SLQ) is called well-posed at $(t,x)$ if $V(t,x)>-\infty$.
A well-posed problem is called {\it attainable} if there is a control $u^*(\cdot)\in\mathcal U_{ad}$ such that $J(t,x;u^*(\cdot))=V(t,x)$. In this case, $u^*(\cdot)$ is called an {\it optimal control} and the corresponding solution of \eqref{system}, $X^*(\cdot)$ is called the
{\it optimal trajectory} (corresponding to $u^*(\cdot)$), and
$(X^*(\cdot), u^*(\cdot))$ is called an {\it optimal pair}.

If $R>0$ and $Q-S^\top R^{-1}S\geq 0$, then $V(t,x)\geq 0$ so that Problem (SLQ) is well-posed for any given $t\geq 0$ and $x\in\mathbb R^n$. If $R>0$ and $Q-S^\top R^{-1}S=0$, then
\begin{multline*} 
J(t,x;u(\cdot))\\
=\mathbb E^{\mathcal F_t}\int_{t}^\infty\big[(S X(s)+Ru(s))^{\top} R^{-1}(S X(s)+R u(s))\big]ds\geq 0, 
\end{multline*}
Clearly, $0$ is a lower bound and it is achieved evidently by the unique optimal control 
$u^{*}(\cdot)=-R^{-1}SX(\cdot)$.
From now on, we focus on the following case. 
\begin{assumption}\label{ass2}
	$R>0$ and $Q-S^\top R^{-1}S>0$.
\end{assumption}


%

The following result is well known; please see Theorem 3.3 in Chapter 4 of \cite{Yong-Zhou-1999} or Theorem 13 in \cite{Rami-Zhou-2000}.

\begin{lemma}\label{lemma-0}
Suppose $P\in\mathcal S^n_{++}$ satisfies the following Lyapunov equation
\begin{multline}
(A+BK)^\top P+P(A+BK)\\
+(C+DK)^\top P(C+DK)\\
+K^\top RK+S^\top K+K^\top S+Q=0
\end{multline}\label{Riccati}
where $$K=-(R+D^\top P D)^{-1}(B^\top P+D^\top PC+S ).$$
Then $u(\cdot)=KX(\cdot)$ is an optimal control of Problem (SLQ) and $V(t,x)=x^\top Px.$ Moreover, we have the Bellman's DP: 
\begin{multline}\label{Bellman}
x^\top Px
=\mathbb E^{\mathcal F_t}\Big\{\int_t^{t+\Delta t}X(s)^\top\Big[ Q+2K^\top S
+K^\top RK\Big]X(s)ds\\
+X(t+\Delta t)^\top PX(t+\Delta t)\Big\},
\end{multline}
for any constant $\Delta t>0$.
\end{lemma} 

Our key observation is that, based on \eqref{Bellman}, to compute the value function $V$ is equivalent to calculate $P$. We only need to know the local state trajectories $X(\cdot)$ on $[t,t+\Delta t]$, therefore it requires us to provide the reasonable online algorithm to solve Problem (SLQ). Indeed, the value of $P$ can be inferred from \eqref{Bellman} by the local trajectories of $X(\cdot)$. 
On the other hand, we can also compute $P$ by the following expression 
\begin{equation}\label{value-f}
x^\top Px=\mathbb E^{\mathcal F_t}\int_t^\infty X(s)^\top\left[ Q+2K^\top S+K^\top RK\right]X(s)ds,
\end{equation} 
which is obtained by sending $\Delta t$ to infinity in \eqref{Bellman}. 
But it requires the entire state trajectories $X(\cdot)$ on $[t,\infty)$.

 At each iteration $i$ ($i=1,2,\cdots$), the state trajectory is denoted by $X^{(i)}$ corresponding to the control law $K^{(i)}$.
Now, we present Algorithm \ref{algorithm} as follows. 

\begin{algorithm}\caption{\small {\it Policy Iteration for Problem (SLQ)}}
1: {\bf Initialization:} Select any stabilizer $K^{(0)}$ for the system \eqref{system}.\\
2: Let $i=0$ and $\varepsilon>0$.\\
3: {\bf do} $\{$\\
4: Obtain local state trajectories $X^{(i)}$ by running system \eqref{system} with $K^{(i)}$ on $[t,t+\Delta t]$.\\
5: {\bf Policy Evaluation} (Reinforcement): Solve $P^{(i+1)}$ from the identity 
\begin{align}
&x^{\top} P^{(i+1)}x-\mathbb E^{\mathcal F_t}\big[X^{(i)}(t+\Delta t)^\top P^{(i+1)}X^{(i)}(t+\Delta t)\big]\nonumber\\&=\mathbb E^{\mathcal F_t}\int_t^{t+\Delta t}X^{(i)}(s)^{\top}\Big[ Q+2K^{(i)\top} S\nonumber\\&~~~\qquad\qquad\qquad\qquad\qquad+K^{(i)\top} RK^{(i)}\Big]X^{(i)}(s)ds.
\label{algorithm 1-evaluation}
\end{align}\\
6: {\bf Policy Improvement} (Update): Update $K^{(i+1)}$ by the formula 
\begin{equation}\label{algorithm 1-improvement}
\begin{aligned}
K^{(i+1)}&=-(R+D^\top P^{(i+1)} D)^{-1}(B^\top P^{(i+1)}\\
&~~~~~~+D^\top P^{(i+1)}C+S ).
\end{aligned} 
\end{equation}

7:  $i\leftarrow i+1$\\
8: $\}$ {\bf until} $\|P^{(i+1)}-P^{(i)}\|<\varepsilon$.\label{algorithm}
\end{algorithm}

Algorithm \ref{algorithm} is an online algorithm based on local state trajectories, reinforced by Policy Evaluation \eqref{algorithm 1-evaluation} and updated by Policy Improvement \eqref{algorithm 1-improvement}. 
Algorithm \ref{algorithm} has three significant advantages over the offline algorithm:
(i) The observation period consisting of an initial time $t\in[0,\infty)$ and a length $\Delta t>0$ can be freely chosen; (ii)  Different from \eqref{value-f} exploring the entire state space on the whole interval $[t,\infty)$, we only need to record local observations of the state on the short period $[t,t+\Delta t]$, which dramatically reduces the computation at each iteration; (iii) This algorithm can be implemented without using the information of $A$ in the system $[A,C;B,D]$, so it is {\em partially model-free}. Especially when $D=\bf O$, Algorithm \ref{algorithm} can be implemented without using the information of $A$ and $C$.

\begin{lemma}\label{lemma-1}
Suppose that Assumption \ref{ass2} holds and the system $[A,C;B,D]$ is stabilizable with $K^{(i)}$. Then solving Policy Evaluation \eqref{algorithm 1-evaluation} in Algorithm \ref{algorithm} is equivalent to solving Lyapunov Recursion
\begin{equation}\label{Lyapunov Recursion}
\begin{aligned}
&(A+BK^{(i)})^\top P^{(i+1)}+P^{(i+1)}(A+BK^{(i)})\\
&~~~+(C+DK^{(i)})^\top P^{(i+1)}(C+DK^{(i)})\\
&~~~+K^{(i)\top} RK^{(i)}+S^\top K^{(i)}+K^{(i)\top} S+Q=0.
\end{aligned}
\end{equation}
\end{lemma}
\begin{proof}
Suppose $K^{(i)}$ is a stabilizer for the system \eqref{system}. 
By Assumption 2.2, 
\[
\begin{aligned}
	&~~~K^{(i)\top} RK^{(i)}+S^\top K^{(i)}+K^{(i)\top} S+Q\\
	&=Q-S^\top R^{-1}S+(RK^{(i)}+S)^\top R^{-1}(RK^{(i)}+S)>0.
	\end{aligned}\] 
By Lemma \ref{lemma-stabilizer}, Lyapunov Recursion \eqref{Lyapunov Recursion} admits a unique solution $P^{(i+1)}\in\mathcal S^n_{++}$. 

Inserting the feedback control $u^{(i)}(\cdot)=K^{(i)}X^{(i)}(\cdot)$ into the system \eqref{system}
and applying It\^o's formula to $X^{(i)}(\cdot)^\top P^{(i+1)}X^{(i)}(\cdot)$, we have
\begin{align}
&\quad~d\left[X^{(i)}(s)^\top P^{(i+1)}X^{(i)}(s)\right]\nonumber\\
&=\bigg\{X^{(i)}(s)^\top\Big[(A+BK^{(i)})^\top P^{(i+1)}+P^{(i+1)}(A+BK^{(i)})\nonumber\\
&~~~+(C+DK^{(i)})^\top P^{(i+1)}(C+DK^{(i)})\Big]X^{(i)}(s)\bigg\}ds\nonumber\\
&~~~+\left\{...\right\}dW(s).\label{ito formula}
\end{align}
 Integrating from $t$ to $t+\Delta t$, we have
\begin{align*}
&\quad~X^{(i)}(t+\Delta t)^\top P^{(i+1)}X^{(i)}(t+\Delta t)-x^\top P^{(i+1)}x\nonumber\\
&=\int_t^{t+\Delta t}X^{(i)} (s)^\top\Big[(A+BK^{(i)})^\top P^{(i+1)}+P^{(i+1)}(A+BK^{(i)})\nonumber\\
&~~~+(C+DK^{(i)})^\top P^{(i+1)}(C+DK^{(i)})\Big]X^{(i)}(s)ds\nonumber\\
&~~~+\int_t^{t+\Delta t}\left\{...\right\}dW(s).
\end{align*}
Since $\mathbb E^{\mathcal F_t}\int_t^{t+\Delta t}\left\{...\right\}dW(s)=0$, 
taking conditional expectation $\mathbb E^{\mathcal F_t}$ on both sides, one gets
\begin{align}
&\quad~\mathbb E^{\mathcal F_t}[X^{(i)}(t+\Delta t)^\top P^{(i+1)}X^{(i)}(t+\Delta t)]-x^\top P^{(i+1)}x\nonumber\\
&=\mathbb E^{\mathcal F_t}\!\!\int_t^{t+\Delta t}\!\!\!\!X^{(i)} (s)^\top\Big[(A+BK^{(i)})^\top P^{(i+1)}+P^{(i+1)}(A+BK^{(i)})\nonumber\\
&~~~+(C+DK^{(i)})^\top P^{(i+1)}(C+DK^{(i)})\Big]X^{(i)}(s)ds.\label{integrating}
\end{align}
From Lyapunov Recursion \eqref{Lyapunov Recursion}, we have
\[
\begin{aligned}
&~~~\mathbb E^{\mathcal F_t}[X^{(i)}(t+\Delta t)^\top P^{(i+1)}X^{(i)}(t+\Delta t)]-x^\top P^{(i+1)}x\\
&=-\mathbb E^{\mathcal F_t}\int_t^{t+\Delta t}\Big\{X^{(i)}(s)^\top\Big[Q+2K^{(i)\top} S\\
&~~~\qquad\qquad\qquad\qquad+K^{(i)\top} RK^{(i)}\Big]X^{(i)}(s)\Big\}ds,
\end{aligned}
\]
which confirms Policy Evaluation \eqref{algorithm 1-evaluation}.

On the other hand, 
if $P^{(i+1)}\in\mathcal S^n$ is the solution of \eqref{algorithm 1-evaluation}, for any constant $\tau > t$, a calculation similar to \eqref{integrating} gives
{\small\begin{align}
&\mathbb E^{\mathcal F_\tau}\int_\tau^{\tau+\Delta t}X^{(i)} (s)^\top\Big[(A+BK^{(i)})^\top P^{(i+1)}+P^{(i+1)}(A+BK^{(i)})\nonumber\\
&~~\qquad+(C+DK^{(i)})^\top P^{(i+1)}(C+DK^{(i)})\Big]X^{(i)}(s)ds\nonumber\\
&+\mathbb E^{\mathcal F_\tau}\int_\tau^{\tau+\Delta t}\Big\{X^{(i)}(s)^\top\Big[Q+2K^{(i)\top} S\nonumber\\
&~\qquad~\qquad~\qquad~\qquad+K^{(i)\top} RK^{(i)}\Big]X^{(i)}(s)\Big\}ds=0.\label{combining}
\end{align}}
Dividing $\Delta t$ on both sides of \eqref{combining}, we see 
 
{\begin{align*}
&\frac{1}{\Delta t}\mathbb E^{\mathcal F_\tau}\int_\tau^{\tau+\Delta t}\Big\{X^{(i)} (s)^\top\Big[(A+BK^{(i)})^\top P^{(i+1)}\\
&+P^{(i+1)}(A+BK^{(i)}) +(C+DK^{(i)})^\top P^{(i+1)}(C+DK^{(i)})\\
&+Q+2K^{(i)\top} S +K^{(i)\top} RK^{(i)}\Big]X^{(i)}(s)\Big\}ds=0.
\end{align*}}
Let $\Delta t\rightarrow 0$ and denote the state at time $\tau$ by $x_{\tau}$, then 
\begin{equation*}\label{Lyapunov x}
\begin{aligned}
&x_\tau^\top\Big[(A+BK^{(i)})^\top P^{(i+1)}+P^{(i+1)}(A+BK^{(i)})\\
&~~~+(C+DK^{(i)})^\top P^{(i+1)}(C+DK^{(i)})\\
&~~~+K^{(i)\top} RK^{(i)}+S^\top K^{(i)}+K^{(i)\top} S+Q\Big]x_\tau=0.
\end{aligned}
\end{equation*}
Because $x_\tau$ can be taken any value in $\mathbb R^n$, Lyapunov Recursion \eqref{Lyapunov Recursion} holds. 
By Lemma \ref{lemma-stabilizer} and 
\[
	K^{(i)\top} RK^{(i)}+S^\top K^{(i)}+K^{(i)\top} S+Q>0,
\]
we have $P^{(i+1)}\in\mathcal S^n_{++}$. 
\end{proof}

By Lemma \ref{lemma-1}, solving Lyapunov Recursion \eqref{Lyapunov Recursion} with Policy Improvement \eqref{algorithm 1-improvement} is equivalent to solving Policy Evaluation \eqref{algorithm 1-evaluation}; that is, they admit the same solution $P^{(i+1)}$ at each iteration. 
The latter has a significant advantage over the former in that it does not necessitate knowing all the system's information. Indeed, the information of coefficient $A$ is  embedded in the state trajectories $X^{(i)}$ online, so we can use the observation of state trajectories to reinforce \eqref{algorithm 1-evaluation} without knowing $A$ in our algorithm. The coefficients $B$, $C$, and $D$ are used to update the control law $K^{(i)}$ in Policy Improvement \eqref{algorithm 1-improvement}. In particular, $C$ is not required to know when $D=\bf O$.

Once initializing a stabilizer $K^{(0)}$ in Algorithm \ref{algorithm}, one first runs the system repeatedly with the control $K^{(i)}$ from the initial state $x$ and records the resultant state trajectories $X^{(i)}$ on interval $[t,t+\Delta t]$ to reinforce the \emph{target function}:
{\small\begin{equation}\label{target function}
\begin{aligned}
&~~~\Delta J^{(i)}(t,t+\Delta t;X^{(i)},K^{(i)})\\
&:=\mathbb E^{\mathcal F_t}\Big\{\int_t^{t+\Delta t}X^{(i)}(s)^{\top}\Big[ Q+2K^{(i)\top} S+K^{(i)\top} RK^{(i)}\Big]X^{(i)}(s)ds\Big\}.
\end{aligned}
\end{equation}}
Then one solves $P^{(i+1)}$ by \eqref{algorithm 1-evaluation} and obtains an updated control $K^{(i+1)}$ by \eqref{algorithm 1-improvement}. This procedure is iterated until it converges. In this procedure, 
$\{K^{(i)}\}_{i=1}^\infty$ should be the stabilizers of the system $[A,C;B,D]$ of adaptive process at each iteration, i.e., it is necessary to require that $K^{(i)}$ is stepwise stable. The following lemma illustrates the stepwise stable property of $K^{(i)}$. 

\begin{theorem}\label{th-2}
Suppose that Assumptions \ref{ass1} and \ref{ass2} hold. Also suppose $K^{(0)}$ is a stabilizer for the system $[A,C;B,D]$. Then all the policies $\{K^{(i)}\}_{i=1}^\infty$ updated by \eqref{algorithm 1-improvement} are stabilizers. Moreover, the solution $P^{(i+1)}\in \mathcal S_{++}^n$ in Algorithm \ref{algorithm} is unique at each step. 
\end{theorem}
\begin{proof}
Because $K^{(0)}$ is a stabilizer for the system $[A,C;B,D]$, by the same argument as the proof of Lemma \ref{lemma-1}, there exists a unique solution $P^{(i+1)}\in\mathcal S^n_{++}$ of Lyapunov Recursion \eqref{Lyapunov Recursion} with $i=0$.

We prove by mathematical induction. Suppose $i\geq 1$, $K^{(i-1)}$ is a stabilizer and $P^{(i)}\in\mathcal S^n_{++}$ is the unique solution of Lyapunov Recursion \eqref{Lyapunov Recursion}.
We now show $K^{(i)}=-(R+D^\top P^{(i)} D)^{-1}(B^\top P^{(i)}+D^\top P^{(i)}C+S )$ is also a stabilizer and $P^{(i+1)}\in\mathcal S^n_{++}$.  To this end, we first notice 
\begin{equation}\label{stabilizer}
\begin{aligned}
&\quad~(A+BK^{(i)})^\top P^{(i)}+P^{(i)}(A+BK^{(i)})\\
&\quad~\quad~+(C+DK^{(i)})^\top P^{(i)}(C+DK^{(i)})\\
&=(A+BK^{(i-1)})^\top P^{(i)}+P^{(i)}(A+BK^{(i-1)})\\
&~~~\qquad+(C+DK^{(i-1)})^\top P^{(i)}(C+DK^{(i-1)})\\
&~~~-\Big[(K^{(i-1)}-K^{(i)})^\top B^\top P^{(i)}+P^{(i)}B(K^{(i-1)}-K^{(i)})\\
&~~~\qquad+(C+DK^{(i-1)})^\top P^{(i)}(C+DK^{(i-1)})\\
&~~~\qquad\qquad-(C+DK^{(i)})^\top P^{(i)}(C+DK^{(i)}) \Big]\\
&=-\big[K^{(i-1)\top} RK^{(i-1)}+S^\top K^{(i-1)}+K^{(i-1)}S+Q\big]\\
&~~~-\Big[(K^{(i-1)}-K^{(i)})^\top B^\top P^{(i)}+P^{(i)}B(K^{(i-1)}-K^{(i)})\\
&~~~\qquad+(K^{(i-1)}-K^{(i)})^\top D^\top P^{(i)}D(K^{(i-1)}-K^{(i)})\\
&~~~\qquad+(K^{(i-1)}-K^{(i)})^\top D^\top P^{(i)}(C+DK^{(i)})\\
&~~~\qquad+(C+DK^{(i)})^\top P^{(i)}D(K^{(i-1)}-K^{(i)})\Big]\\
&=-\big[K^{(i-1)\top} RK^{(i-1)}+S^\top K^{(i-1)}+K^{(i-1)\top}S+Q\big]\\
&~~~\qquad-(K^{(i-1)}-K^{(i)})^\top D^\top P^{(i)}D(K^{(i-1)}-K^{(i)})\\
&~~~-\Big[(K^{(i-1)}-K^{(i)})^\top \big[B^\top P^{(i)}+D^\top P^{(i)}(C+DK^{(i)})\big]\\
&~~~+\big[P^{(i)}B+(C+DK^{(i)})^\top P^{(i)}D\big]^\top(K^{(i-1)}-K^{(i)})\Big]. 
\end{aligned}
\end{equation}
From Policy Improvement \eqref{algorithm 1-improvement}, 
\begin{equation*}\label{K-i+1}
B^\top P^{(i)}+D^\top P^{(i)}C=-(R+D^\top P^{(i)}D)K^{(i)}-S.
\end{equation*}
%
Plugging this into \eqref{stabilizer} and using $Q-S^\top R^{-1}S>0$, we obtain
\begin{equation*}\label{stabilizer2}
\begin{aligned}
&\quad~(A+BK^{(i)})^\top P^{(i)}+P^{(i)}(A+BK^{(i)})\\
&\quad\quad+(C+DK^{(i)})^\top P^{(i)}(C+DK^{(i)})\\
&=-\big[K^{(i)} RK^{(i)}+S^\top K^{(i)}+K^{(i)\top}S+Q\big]\\
&\quad\quad-(K^{(i-1)}-K^{(i)})^\top(R+ D^\top P^{(i)}D)(K^{(i-1)}-K^{(i)})\\
&=-\big[Q-S^\top R^{-1}S+(K^{(i)}+R^{-1}S)^\top R(K^{(i)}+R^{-1}S)\big]\\
&\quad\quad-(K^{(i-1)}-K^{(i)})^\top (R+D^\top P^{(i)}D)(K^{(i-1)}-K^{(i)})\\
&<0.
\end{aligned}
\end{equation*} 
So $K^{(i)}$ is a stabilizer by Lemma \ref{lemma-stabilizer}. Moreover, Lyapunov Recursion \eqref{Lyapunov Recursion} 
admits a unique solution $P^{(i+1)}\in\mathcal S^n_{++}$ since 
\begin{multline*}
	K^{(i)\top} RK^{(i)}+S^\top K^{(i)}+K^{(i)\top} S+Q\\=Q-S^\top R^{-1}S+(RK^{(i)}+S)^\top R^{-1}(RK^{(i)}+S)>0.
\end{multline*} 
From Lemma \ref{lemma-1}, Lyapunov Recursion \eqref{Lyapunov Recursion} is equivalent to Policy Evaluation \eqref{algorithm 1-evaluation}, so $P^{(i+1)}\in\mathcal S^n_{++}$ is the unique solution in Algorithm \ref{algorithm}.
\end{proof} 

Now, we prove the convergence of Algorithm \ref{algorithm}. 

\begin{theorem}\label{theorem 1}
The iteration $\{P^{(i)}\}_{i=1}^{\infty}$ of Algorithm \ref{algorithm} converges to the solution $P\in\mathcal S^n_{++}$ of the following SARE:
\begin{equation}\label{SARE}
\begin{aligned}
&A^\top P+PA^\top+C^\top PC+Q\\
&~~~-(PB+C^\top PD+S^\top)(R+D^\top P D)^{-1}\\
&\qquad\qquad\times(B^\top P+D^\top PC+S )=0.
\end{aligned}
\end{equation}
Also, the optimal control of Problem (SLQ) is 
\begin{equation}\label{optimal u}
u^*=-(R+D^\top P D)^{-1}(B^\top P+D^\top PC+S )X^*,
\end{equation}
where $X^*(\cdot)$ is determined by
\[
\left\{
\begin{aligned}
dX^*(s)&=\left(A+BK\right)X^*(s)ds\\
&~~~+\left(C+DK\right)X^*(s)dW(s),~~~s\geq t,\\
X^*(t)&=x,
\end{aligned}
\right.
\] 
with \begin{equation}\label{K}
K=-(R+D^\top P D)^{-1}(B^\top P+D^\top PC+S ).
\end{equation}Moreover, $K$ is a stabilizer of the system $[A,C;B,D]$.
\end{theorem}

\begin{proof}
From Lemma \ref{lemma-1}, Algorithm \ref{algorithm} is equivalent to Lyapunov Recursion \eqref{Lyapunov Recursion} with Policy Improvement \eqref{algorithm 1-improvement}. We now prove $\{P^{(i)}\}_{i=1}^{\infty}$ in \eqref{Lyapunov Recursion} combining with \eqref{algorithm 1-improvement} 
converges to the solution $P$ of SARE \eqref{SARE}. Note $P^{(i+1)}$ satisfies Lyapunov Recursion \eqref{Lyapunov Recursion}. 
Denote
$\Delta P^{(i+1)}=P^{(i)}-P^{(i+1)}$, and $\Delta K^{(i)}=K^{(i-1)}-K^{(i)}$ for $i=1,2,\cdots$.
Then 
\begin{equation}\label{Pi-Pi+1}
\begin{aligned}
0&=(A+BK^{(i-1)})^\top P^{(i)}+P^{(i)}(A+BK^{(i-1)})\\
&~~~~~~~~+(C+DK^{(i-1)})^\top P^{(i)}(C+DK^{(i-1)})\\
&~~~~~~~~~~~~~+K^{(i-1)^\top} RK^{(i-1)}+S^\top K^{(i-1)}+K^{(i-1)\top}S \\
&~~~-\big[(A+BK^{(i)})^\top P^{(i+1)}+P^{(i+1)}(A+BK^{(i)})\\
&~~~~~~~~+(C+DK^{(i)})^\top P^{(i+1)}(C+DK^{(i)})\\
&~~~~~~~~~~~~~+K^{(i)\top} RK^{(i)}+S^\top K^{(i)}+K^{(i)\top}S\big] \\
&=(A+BK^{(i)})^\top\Delta P^{(i+1)}+\Delta P^{(i+1)}(A+BK^{(i)})\\
&~~~~~~~~+(C+DK^{(i)})^\top \Delta P^{(i+1)}(C+DK^{(i)})\\
&~~~+\Delta K^{(i-1)^\top} B^\top P^{(i)}+P^{(i)}B\Delta K^{(i-1)}\\
&~~~~~~~~+(C+DK^{(i-1)})^\top P^{(i)}(C+DK^{(i-1)})\\
&~~~-(C+DK^{(i)})^\top P^{(i)}(C+DK^{(i)})\\
&~~~~~~~~+K^{(i-1)^\top} RK^{(i-1)}-K^{(i)\top} RK^{(i)}\\
&~~~~~~~~~~~~+S^\top\Delta K^{(i)}+\Delta K^{(i)}S.\\
\end{aligned}
\end{equation}
It follows from Policy Improvement \eqref{algorithm 1-improvement} that we have 
\begin{equation}\label{C}
\begin{aligned}
&~~~~(C+DK^{(i-1)})^\top P^{(i)}(C+DK^{(i-1)})\\
&~~~~~~~~~~~-(C+DK^{(i)})^\top P^{(i)}(C+DK^{(i)})\\
&=\Delta K^{(i)^\top} D^\top P^{(i)}D\Delta K^{(i)}\\
&~~~~~~~~~~~+\Delta K^{(i)^\top} D^\top P^{(i)}(C+DK^{(i)})\\
&~~~~~~~~~~~~~~~~~+(C+DK^{(i)})^\top P^{(i)}D\Delta K^{(i)}.
\end{aligned}
\end{equation}
Note 
\begin{equation}\label{Ki}
\begin{aligned} K^{(i-1)^\top} RK^{(i-1)}&-K^{(i)\top} RK^{(i)}=\Delta K^{(i)^\top} R\Delta K^{(i)}\\&~~~+\Delta K^{(i)^\top} RK^{(i)}+K^{(i)\top} R\Delta K^{(i)}.
\end{aligned}
\end{equation}
Combining \eqref{Pi-Pi+1}-\eqref{Ki}, we deduce 
{\small\begin{equation}
\begin{aligned}
&~(A+BK^{(i)})^\top\Delta P^{(i+1)}+\Delta P^{(i+1)}(A+BK^{(i)})\\
&\quad+(C+DK^{(i)})^\top \Delta P^{(i+1)}(C+DK^{(i)})\\
&\quad\quad\quad+\Delta K^{(i)^\top}(R+D^\top P^{(i)}D)\Delta K^{(i)}\\
&+\Delta K^{(i)^\top}\big[B^\top P^{(i)}+D^\top P^{(i)}C+S+(R+D^\top P^{(i)}D)K^{(i)}\big]\\
&+\big[B^\top P^{(i)}+D^\top P^{(i)}C+S+(R+D^\top P^{(i)}D)K^{(i)}\big]^\top \Delta K^{(i)}\\&=0.
\end{aligned}
\end{equation}}
By \eqref{algorithm 1-improvement}, we have
$$-(R+D^\top P^{(i)}D)K^{(i)}=B^\top P^{(i)}+D^\top P^{(i)}C+S, $$
so 
\begin{equation}\label{Delta}
\begin{aligned}
&(A+BK^{(i)})^\top\Delta P^{(i+1)}+\Delta P^{(i+1)}(A+BK^{(i)})\\
&~~~+(C+DK^{(i)})^\top \Delta P^{(i+1)}(C+DK^{(i)})\\
&~~~+\Delta K^{(i)^\top}(R+D^\top P^{(i)}D)\Delta K^{(i)}=0.
\end{aligned}
\end{equation}
Since $K^{(i)}$ is a stabilizer of the system \eqref{system} and $\Delta K^{(i)^\top}(R+D^\top P^{(i)}D)\Delta K^{(i)}\geq 0$, Lyapunov equation \eqref{Delta} admits a unique solution $\Delta P^{(i+1)}\geq 0$ by Lemma \ref{lemma-stabilizer}. Therefore, $\{P^{(i)}\}_{i=1}^{\infty}$ is monotonically decreasing. Notice $P^{(i)}> 0$, so $\{P^{(i)}\}_{i=1}^{\infty}$ converges to some $P\geq 0$. 

Next, we prove that $P$ is the solution of SARE \eqref{SARE}.
When $i\rightarrow\infty$, 
\begin{equation*}
\begin{aligned}
	&(R+D^\top P^{(i+1)} D)^{-1}(B^\top P^{(i+1)}+D^\top P^{(i+1)}C+S )\\
	&~~~\rightarrow (R+D^\top P D)^{-1}(B^\top P+D^\top PC+S ),
\end{aligned} 
\end{equation*}
which means that $\{K^{(i)}\}_{i=1}^{\infty}$ converges to $K$ given by \eqref{K}. Moreover, $(P,K)$ satisfies 
\begin{equation}\label{Lyapunov K}
\begin{aligned}
&(A+BK)^\top P+P(A+BK)+(C+DK)^\top P(C+DK)\\
&~~\qquad\qquad+K^\top RK+S^\top K+KS+Q=0.
\end{aligned}
\end{equation}
Since $K^\top RK+S^\top K+KS+Q>0$,
we get 
\[
(A+BK)^\top P+P(A+BK)+(C+DK)^\top P(C+DK)<0,
\]
which implies that $K$ is a stabilizer of the system \eqref{system}. By \eqref{Lyapunov K} and Lemma \ref{lemma-stabilizer}, $P\in\mathcal S^n_{++}$. Moreover, when plugging \eqref{K} into \eqref{Lyapunov K}, \eqref{Lyapunov K}  becomes SARE \eqref{SARE}. From Theorem 13 in \cite{Rami-Zhou-2000}, we see that \eqref{optimal u} is the unique optimal control. \end{proof}

\section{Online Implementation of Partially Model-Free RL Algorithm}\label{Implementation}
\subsection{Online Implementation}
In this section, we illustrate the implementation of Algorithm \ref{algorithm} in detail. Since there are $N:=\frac{n(n+1)}{2}$ independent parameters in the positive definite matrix $P^{(i+1)}$, we need to observe state along trajectories at least $N$ intervals $[t_j,t_j+\Delta t_j]$ with $j=1,2,\cdots,N$ on $[0,\infty)$ to reinforce target function $\Delta J^{(i)}(t_j,t_j+\Delta t_{j}; X^{(i)},K^{(i)})$ defined by \eqref{target function} with $j=1,2,\cdots,N$. From Policy Evaluation \eqref{algorithm 1-evaluation} in Algorithm \ref{algorithm}, for initial state $x_{t_j}$ at time $t_j$, one needs to solve a set of simultaneous equations
\begin{equation}\label{k equations}
\begin{aligned}
&x_{t_j}^\top P^{(i+1)}x_{t_j}\!\!-\mathbb E^{\mathcal F_{t_j}}\big[X^{(i)}(t_j+\Delta t_j)^\top P^{(i+1)}X^{(i)}(t_j+\Delta t_j)\big]\\
&=\Delta J^{(i)}(t_j,t_j+\Delta t_j; X^{(i)},K^{(i)}) \end{aligned}
\end{equation}
with $j=1,2,\cdots,N$ at each iteration $i$. Sometimes, we suppress $X^{(i)}$ and $K^{(i)}$ in target function \eqref{target function} to avoid heavy notation. 

We will use vectorization and Kronecker product theory to solve the above system \eqref{k equations}; see \cite{Murray-Cox-Lendaris-Saeks-2002} for details. 
Define $\vec (M)$ for $M\in\mathbb R^{n\times m}$ as a vectorization map from a matrix into an $nm$-dimensional column vector for compact representations, which stacks the columns of $M$ on top of one another. 
For example, \[\vec{
\left(\begin{bmatrix}
a_{11} & a_{12} \\ a_{21} & a_{22}\\ a_{31} & a_{32}
\end{bmatrix}\right)
}=(a_{11},a_{21},a_{31},a_{12},a_{22},a_{32})^{\top}.
\]

Let $A\otimes B$ be a Kronecker product of matrices $A$ and $B$, then we have $\vec (ABC)=(C^\top\otimes A)\vec (B)$. The set of equations \eqref{k equations} is transformed to 
{\tiny \begin{align}
\left(\left[ \begin{array}{cc} x^\top_{t_1}\otimes x^\top_{t_1} \\ x^\top_{t_2}\otimes x^\top_{t_2}\\
\vdots\\
x^\top_{t_N}\otimes x^\top_{t_N} \end{array}
\right] -\left[ \begin{array}{cc} \mathbb E ^{\mathcal F_{t_1}}[X^{(i)}(t_1+\Delta t_1)^\top\otimes X^{(i)}(t_1+\Delta t_1)^\top] \\ \mathbb E ^{\mathcal F_{t_2}}[X^{(i)}(t_2+\Delta t_2)^\top\otimes X^{(i)}(t_2+\Delta t_2)^\top] \\
\vdots\\
\mathbb E ^{\mathcal F_{t_N}}[X^{(i)}(t_N+\Delta t_N)^\top\otimes X^{(i)}(t_N+\Delta t_N)^\top] \end{array}
\right] \right)\nonumber\\
\times\;\vec (P^{(i+1)})=\left[\begin{array}{cc} \Delta J^{(i)}(t_1,t_1+\Delta t)\\ \Delta J^{(i)}(t_2,t_2+\Delta t_2)\\
\vdots\\
\Delta J^{(i)}(t_N,t_N+\Delta t_N) \end{array}
\right].\label{N-equations}
\end{align}}
Denote 
\[
\Delta X^{(i)}_j=x^\top_{t_j}\otimes x^\top_{t_j}-\mathbb E ^{\mathcal F_{t_j}}[X^{(i)}(t_j+\Delta t_j)^\top\otimes X^{(i)}(t_j+\Delta t_j)^\top],
\]
and
\[
\mathbb X^{(i)}=\left[ \begin{array}{cc} \Delta X^{(i)}_1 \\ \Delta X^{(i)}_2\\
\vdots\\
\Delta X^{(i)}_N\end{array}
\right],~~~~~\mathbb J^{(i)}=\left[\begin{array}{cc} \Delta J^{(i)}(t_1,t_1+\Delta t)\\ \Delta J^{(i)}(t_2,t_2+\Delta t)\\
\vdots\\
\Delta J^{(i)}(t_N,t_N+\Delta t_N) \end{array}
\right], \]
then \eqref{N-equations} is rewritten as
\begin{equation}\label{P-(i+1)}
\mathbb X^{(i)}\vec (P^{(i+1)})=\mathbb J^{(i)}.
\end{equation}

 In practice, we derive the expectation in \eqref{N-equations} by calculating the mean-value by $K$ sample paths $X_k$, $k=1,2,\cdots,K$. Precisely, we calculate 
\[
\begin{aligned}
	\mathbb E ^{\mathcal F_{t_j}}[X^{(i)}(t_j+\Delta t_j)^\top\otimes X^{(i)}(t_j+\Delta t_j)^\top]\\=\frac{1}{K}\sum_{k=1}^{K}[X_k^{(i)}(t_j+\Delta t_j)^\top\otimes X_k^{(i)}(t_j+\Delta t_j)^\top]\end{aligned}\] by the sampled data at terminal time $t_j+\Delta t_j$. In \eqref{P-(i+1)}, if we get $K$ sample paths with the data sampled at discrete time $t_l$ $(t_j\leq t_l\leq t_j+\Delta t_j, l=1,2, \cdots, L)$, we calculate $\Delta J^{(i)}$ in $\mathbb J^{(i)}$ as
\[
\begin{aligned}
&~~~\Delta J^{(i)}(t_j, t_j+\Delta t_j)\\&=\frac{1}{K}\sum_{k=1}^{K}\Big[\sum_{l=1}^L X_k^{(i)}(t_l)^{\top}\big[ Q+2K^{(i)\top} S+K^{(i)\top} RK^{(i)}\big]X_k^{(i)}(t_l)\Big].\end{aligned}\]

Moreover, we define an operator $\vec^+(P)$ for $P\in\mathcal S^{n}$, which maps $P$ into an $N$-dimensional vector by stacking the columns corresponding to the diagonal and lower triangular parts of $P$ on top of one another where the off-diagonal terms of $P$ are double. For example, 
\[\vec^+{
\left(\begin{bmatrix}
p_{11} & p_{12} & p_{13}\\ p_{12} & p_{22} & p_{23}\\ p_{13} & p_{23} & p_{33}
\end{bmatrix}\right)
}=(p_{11},2p_{12},2p_{13},p_{22}, 2p_{23},p_{33})^{\top}.
\]
Similar to \cite{Murray-Cox-Lendaris-Saeks-2002}, there exists a matrix $\mathcal T\in\mathbb R^{n^2\times N}$ with $\mathrm{rank}(\mathcal T)=N$ such that $\vec (P)=\mathcal T\vec^+(P)$ for any $P\in\mathcal S^n$. Then equation \eqref{P-(i+1)} becomes 
\begin{equation}\label{P-(i+1)-2}
\big(\mathbb X^{(i)}\mathcal T\big) \vec^+(P^{(i+1)})=\mathbb J^{(i)}.
\end{equation}To get $\vec^+(P^{(i+1)})$ from \eqref{P-(i+1)-2}, one must chose enough trajectories $X^{(i)}$ on intervals $[t_j,t_j+\Delta t_j]$ with $j=1,2,\cdots,N$ such that
\begin{equation}\label{vecP}
\vec^+(P^{(i+1)})=(\mathbb X^{(i)}\mathcal T)^{-1}\mathbb J^{(i)}.
\end{equation}
Finally, we obtain $P^{(i+1)}$ by taking the inverse map of $\vec^+(\cdot)$.

\subsection{Numerical Examples}
This section presents two numerical examples with dimensions 2 and 5, respectively. Firstly, let $n=2$ and $m=1$, and set 
\[
\begin{aligned}
A=\begin{bmatrix} 0.3 & 0.7 \\ -0.9 & 0.5 \end{bmatrix}, ~
B=\begin{bmatrix} 0.2 \\ 0 \end{bmatrix}, ~
C=\begin{bmatrix} 0.05 & 0.03 \\ 0.05 & 0.02 \end{bmatrix}, ~
D=\begin{bmatrix} 0.05 \\ 0.06 \end{bmatrix},\quad
\end{aligned}
\] 
and $x=(2, 3)^\top $. 
The coefficients in cost functional are chosen as
\[
 Q=\begin{bmatrix} 3 & 0 \\ 0 & 2 \end{bmatrix}, \quad
S={\bf O},\quad
 R=1.25. 
\]


By implementing Algorithm \ref{algorithm}, we only need state information without knowing the coefficient $A$ in the system. In the beginning, we initialize the stabilizer $K^{(0)}=(-8.3809, 7.4036)$. Here, $K^{(0)}$ can be chosen arbitrarily in $\mathcal X([A,C;B,D])$. Then, we read the data of state trajectory $X=(X_1, X_2)^\top$, which is presented in Fig. \ref{fig:subfig} (a).
%

%
Because there are $N=\frac{n(n+1)}{2}=3$ independent parameters of $P$ in this example, we need to select $3$ intervals $[0,1], [1,2]$ and $[2,3]$ to reinforce target function $\Delta J^{(i)}(t_j,t_j+\Delta t_j;X^{(i)},K^{(i)})$ $(t_j=0,1,2, \Delta t_j=1)$ defined by \eqref{target function}. By implementing Algorithm \ref{algorithm}, we calculate $P^{(i+1)}$ by \eqref{vecP} and obtain 

$$P^*=\begin{bmatrix} 61.1422 & -35.7578 \\ -35.7578 & 81.6610 \end{bmatrix}$$
after $12$ iterations in { $5$ seconds}, please see details in Fig. \ref{fig:subfig} (b). 

We denote the left hand side of \eqref{SARE} as
\begin{multline}\label{calR}
	\mathcal R(P)=A^\top P+PA^\top+C^\top PC+Q\\-(PB+C^\top PD+S^\top)(R+D^\top P D)^{-1}(B^\top P+D^\top PC+S ).
\end{multline} 
It is used to measure the distance from $P^*$ to the real solution of SARE \eqref{SARE}. 
Insert $P^*$ in to \eqref{calR}, we have 
$$\mathcal R(P^*)=10^{-4}
\cdot\begin{bmatrix} 0.6829 & 0.1212 \\ 0.1212 & -0.7346 \end{bmatrix}.$$ 
Moreover, $\|\mathcal R(P^*)\|=1.0175\times10^{-4}$. We choose the optimal control as $u^*=K^*X^*=-(R+D^\top P^*D)^{-1}(B^\top P^*+D^\top P^*C)X^*=(-8.3854, 4.7642)X^*$. The optimal trajectory $X^*=(X^*_1,X^*_2)^\top$ under the optimal control $u^*=K^*X^*$ is presented in Fig. \ref{fig:subfig} (c).


 Now, we compare our result with an model-based approach, which involves two steps: first obtain an estimation of $A$ and then solve the Riccati equation by the SDP method in \cite{Rami-Zhou-2000}. Firstly, we use the least-square method to approximate $A$ by the trajectory data on the time interval $[0,n^2]$ based on $x_{k+1} = x_k + \int_{t_k}^{t_{k+1}}(A+BK)X(s)ds+\int_{t_k}^{t_{k+1}}(C+DK)X(s)dW(s)$. The estimation procedure is as follows. 
\begin{enumerate}
	\item Select $K = -(D^\top D)^{-1}D^\top C$;
	\item Read the data $\{x_{k}\}_{k=0}^{n^2}$ at time $t_k = \frac{k}{n^2}, k = 0, 1,2,3, \cdots, n^2$;
	\item Define $y_k = (x_{k+1}-x_k)/(t_{k+1}-t_k)$, and note that \\$~~~~~ Y=(y_0,\cdots,y_{n^2-1}), ~~~ X=(x_0,\cdots,x_{n^2-1})$;
	\item Estimate $\widehat A= YX^\top(XX^\top)^{-1} - BK$.
\end{enumerate}

%
{ By the above procedure, we obtain an approximation of $A$ as 
$$\widehat A=\begin{bmatrix} 0.2984 & 0.7015 \\ -0.9036 & 0.4988 \end{bmatrix}$$ in $10$ seconds}, 
please see details in Fig. \ref{fig:subfig} (d). 
Secondly, using the SDP method in \cite{Rami-Zhou-2000}, we obtain $$\widehat P^*=\begin{bmatrix} 60.8975 & -35.4013 \\ -35.4013 & 80.8154 \end{bmatrix}$$ in $2$ seconds with $$\mathcal R(\widehat P^*)=\begin{bmatrix} 0.2856 & -0.1789 \\ -0.1789 & 0.2629 \end{bmatrix}$$ and $\|\mathcal R(\widehat P^*)\|=0.4633$. Comparing the proposed partially model-free method in this paper to the above model-based method, the former is more effective than the latter in time consumption and accuracy.

\begin{figure*}[htbp]
\centering

\begin{minipage}[b]{1\textwidth}
\subfigure[]{\includegraphics[width=0.33\textwidth]{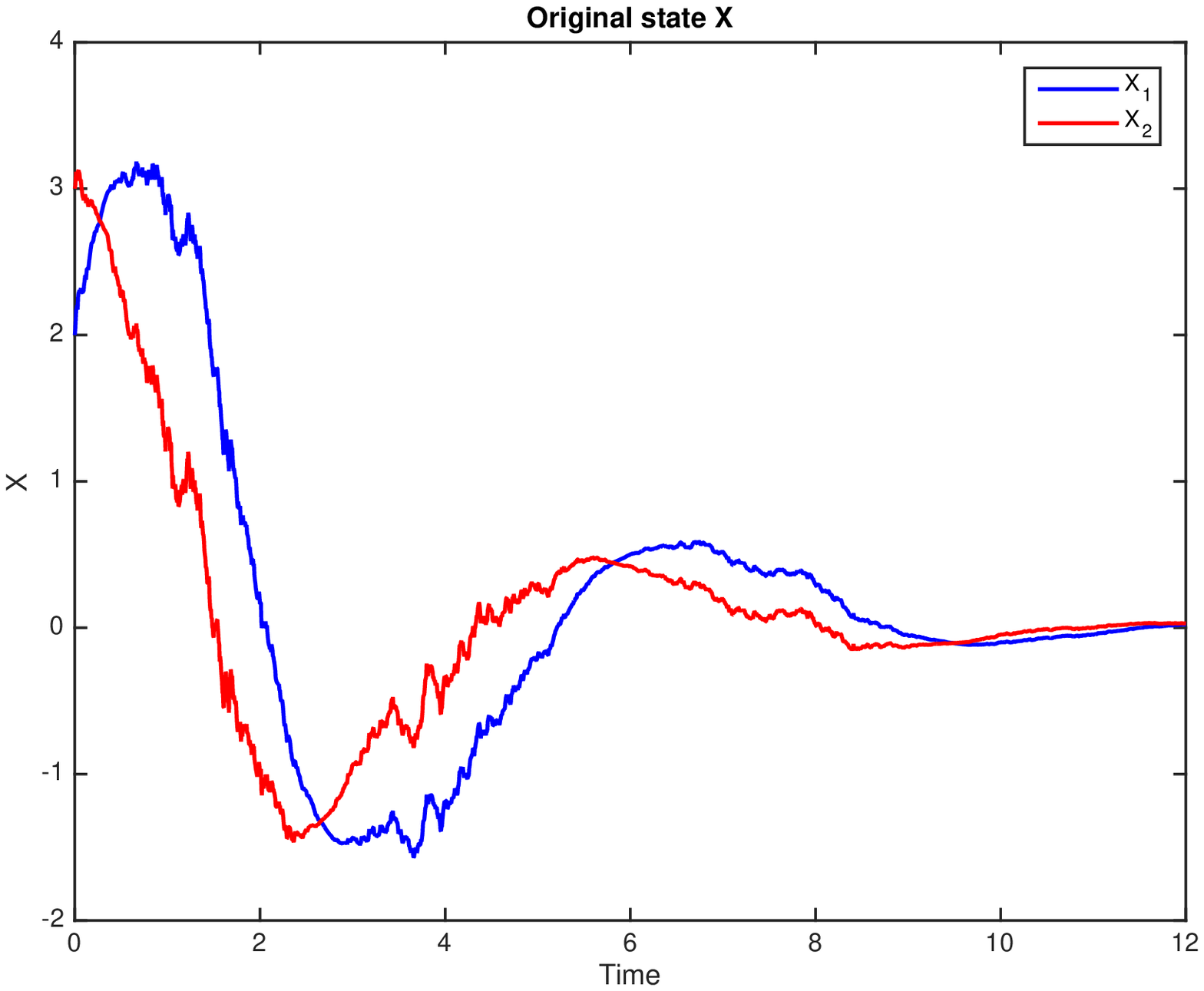}}\hspace{0.02in}
\subfigure[]{\includegraphics[width=0.33\textwidth]{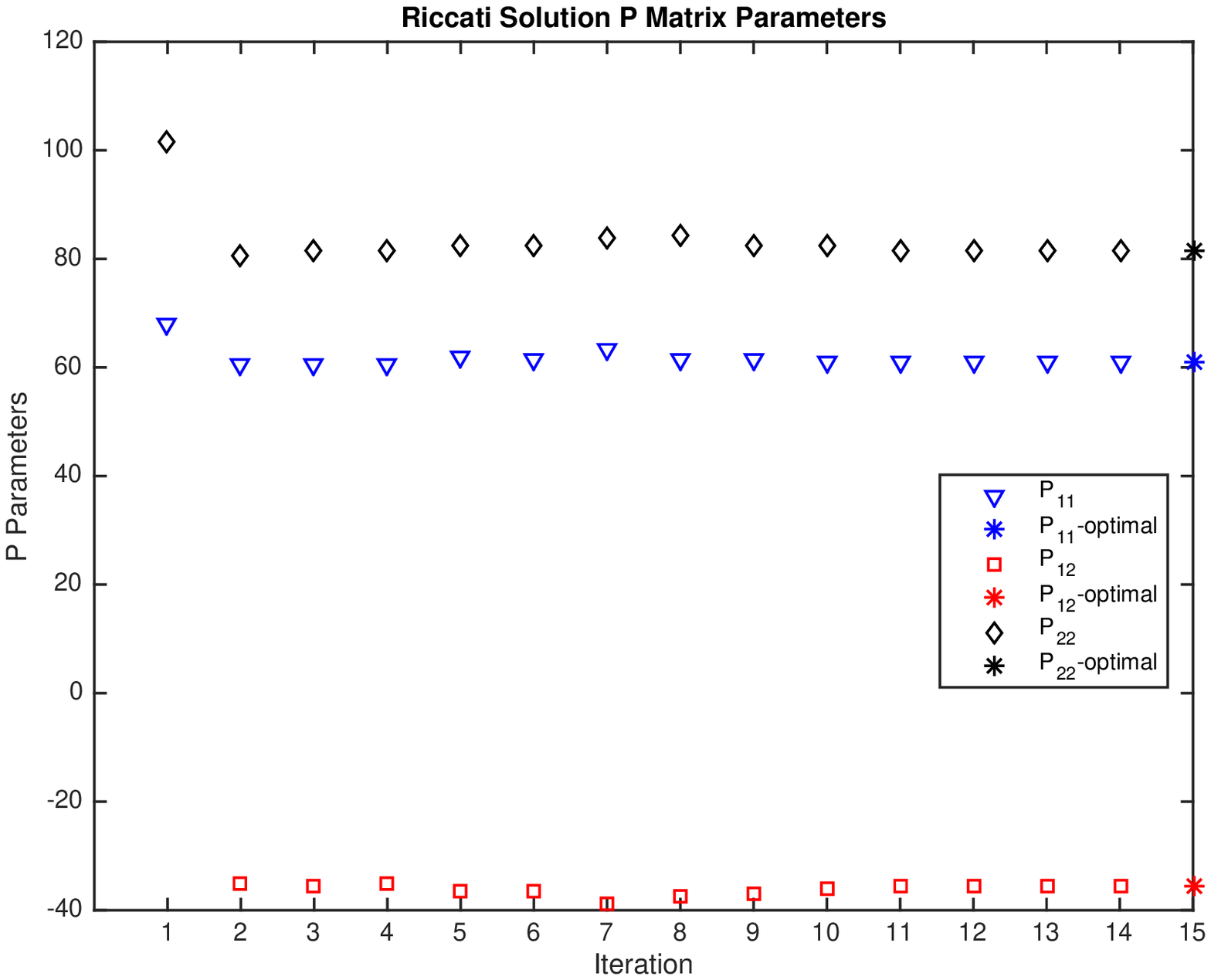} }\vspace{0.02in}
\subfigure[]{\includegraphics[width=0.33\textwidth]{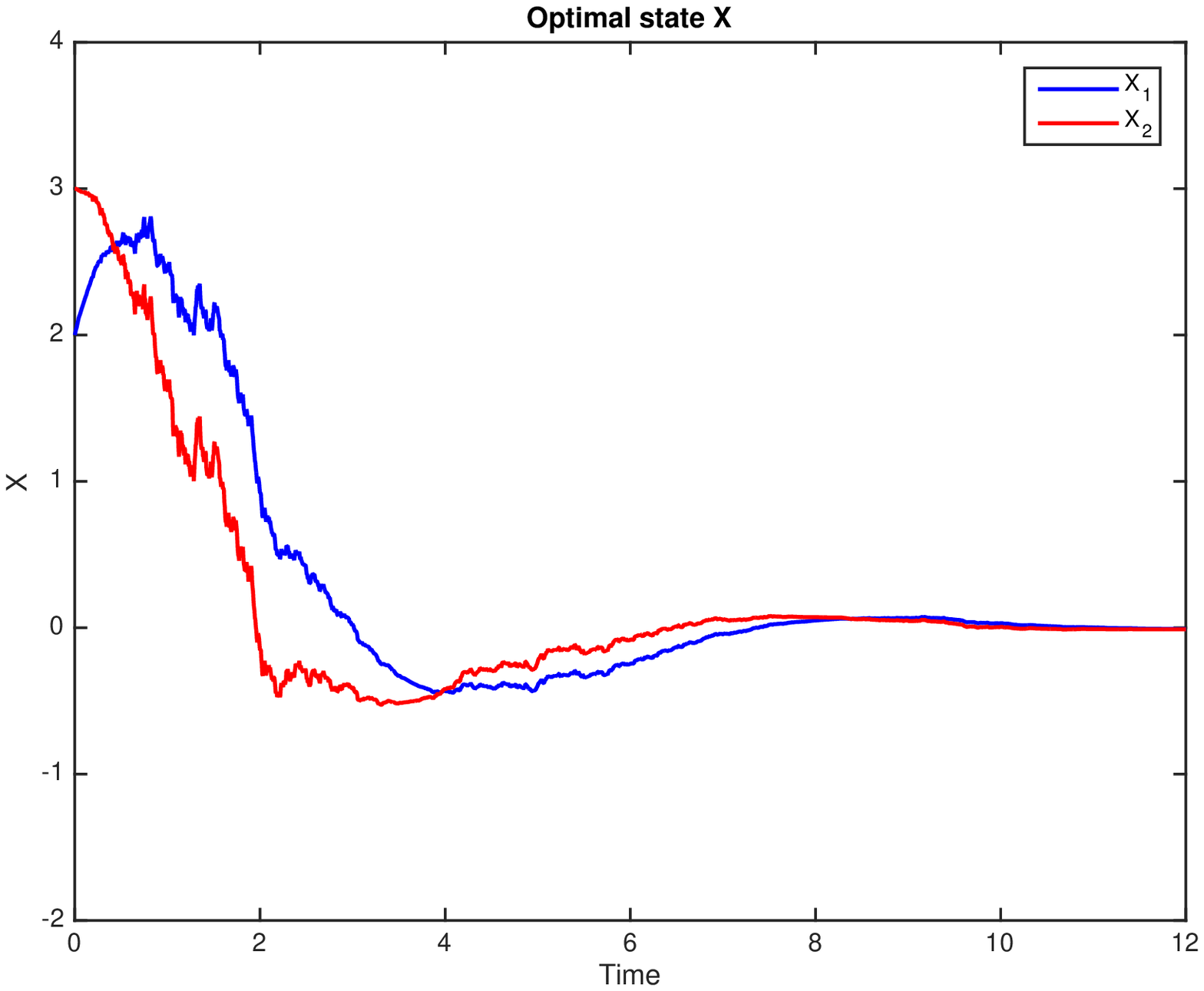} }\vspace{0.02in} 
\end{minipage}

\begin{minipage}[b]{1\textwidth}
\subfigure[]{\includegraphics[width=0.33\textwidth]{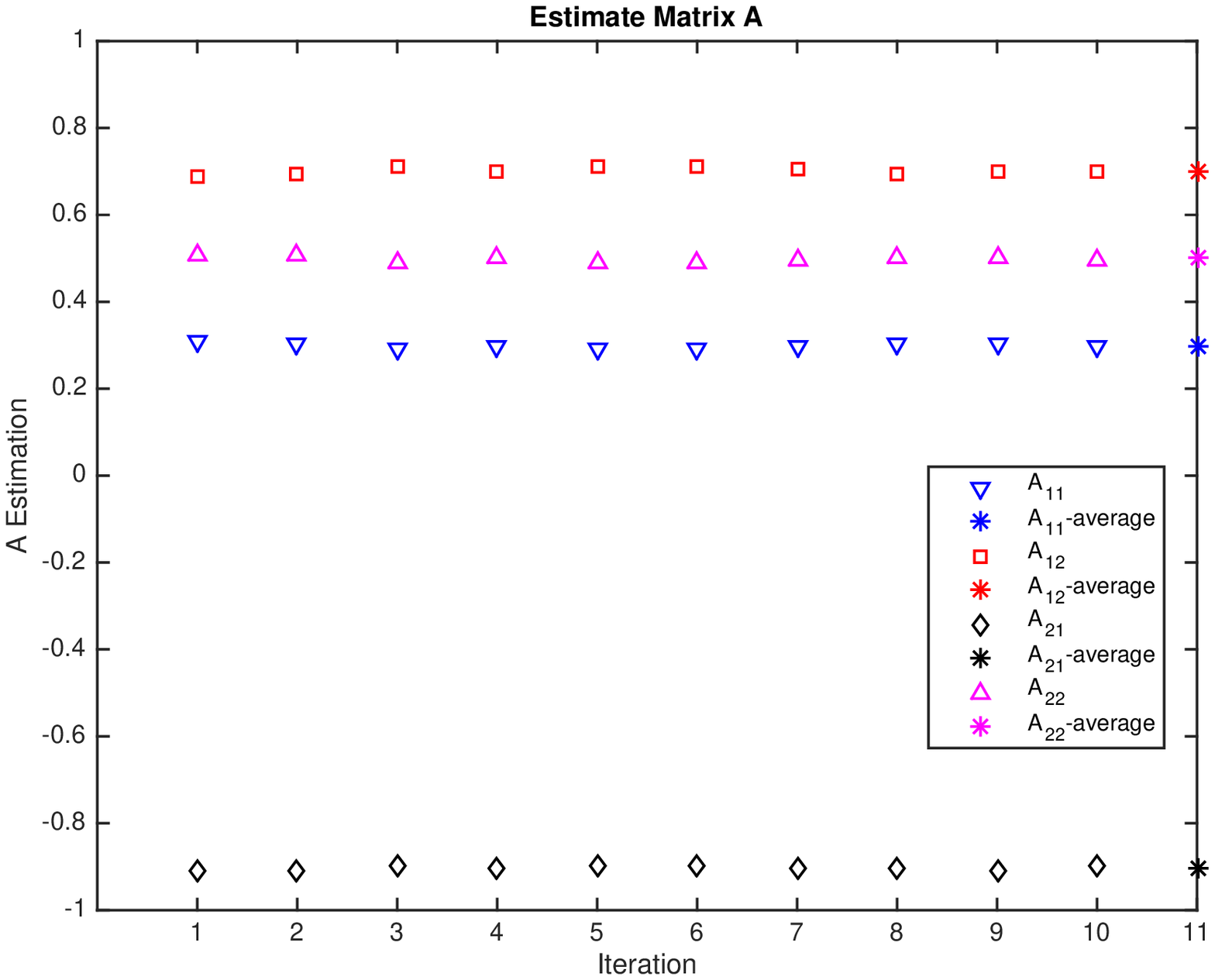}}\hspace{0.02in}\label{aa}
\subfigure[]{\includegraphics[width=0.33\textwidth]{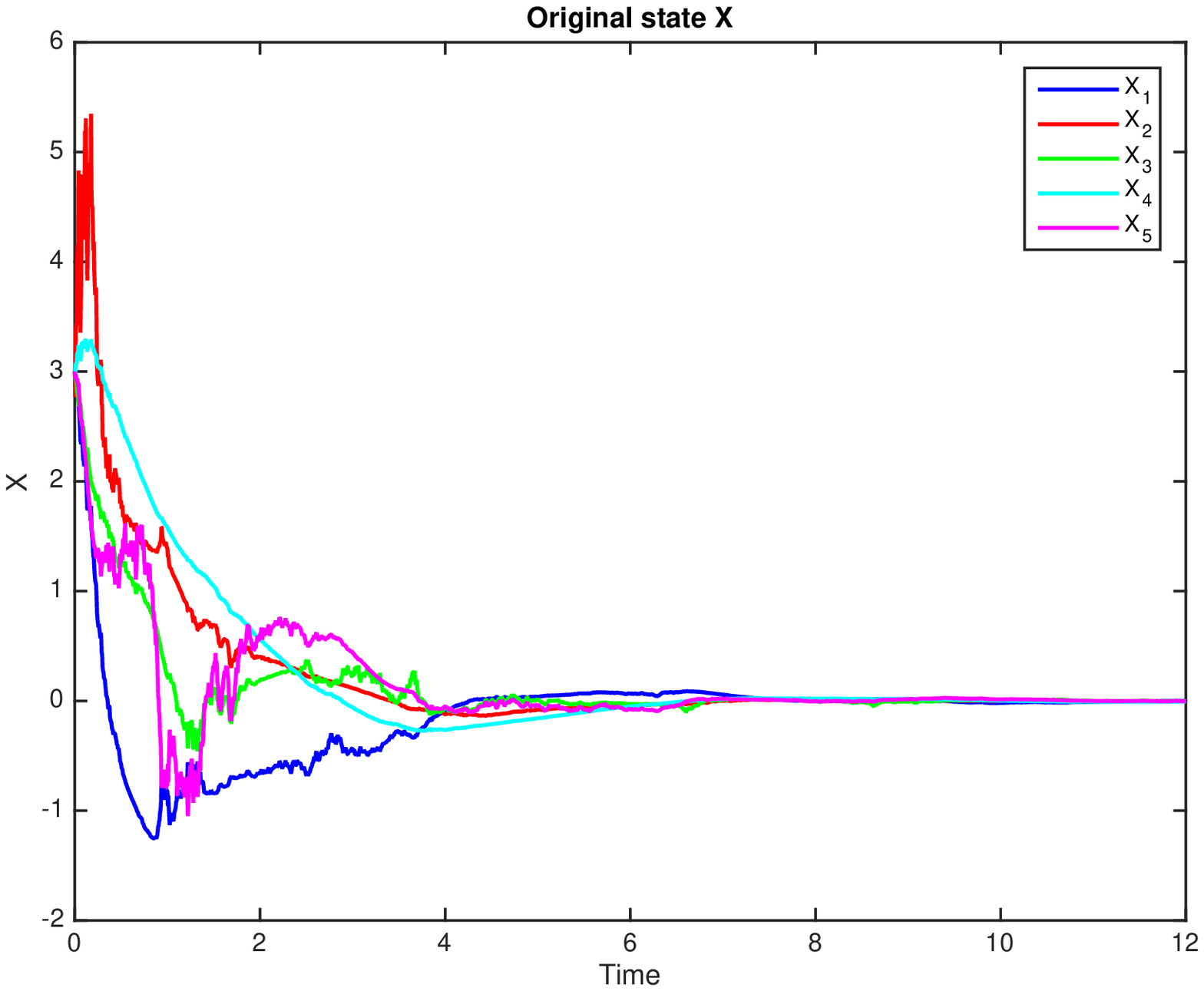} }\vspace{0.02in}
\subfigure[]{\includegraphics[width=0.33\textwidth]{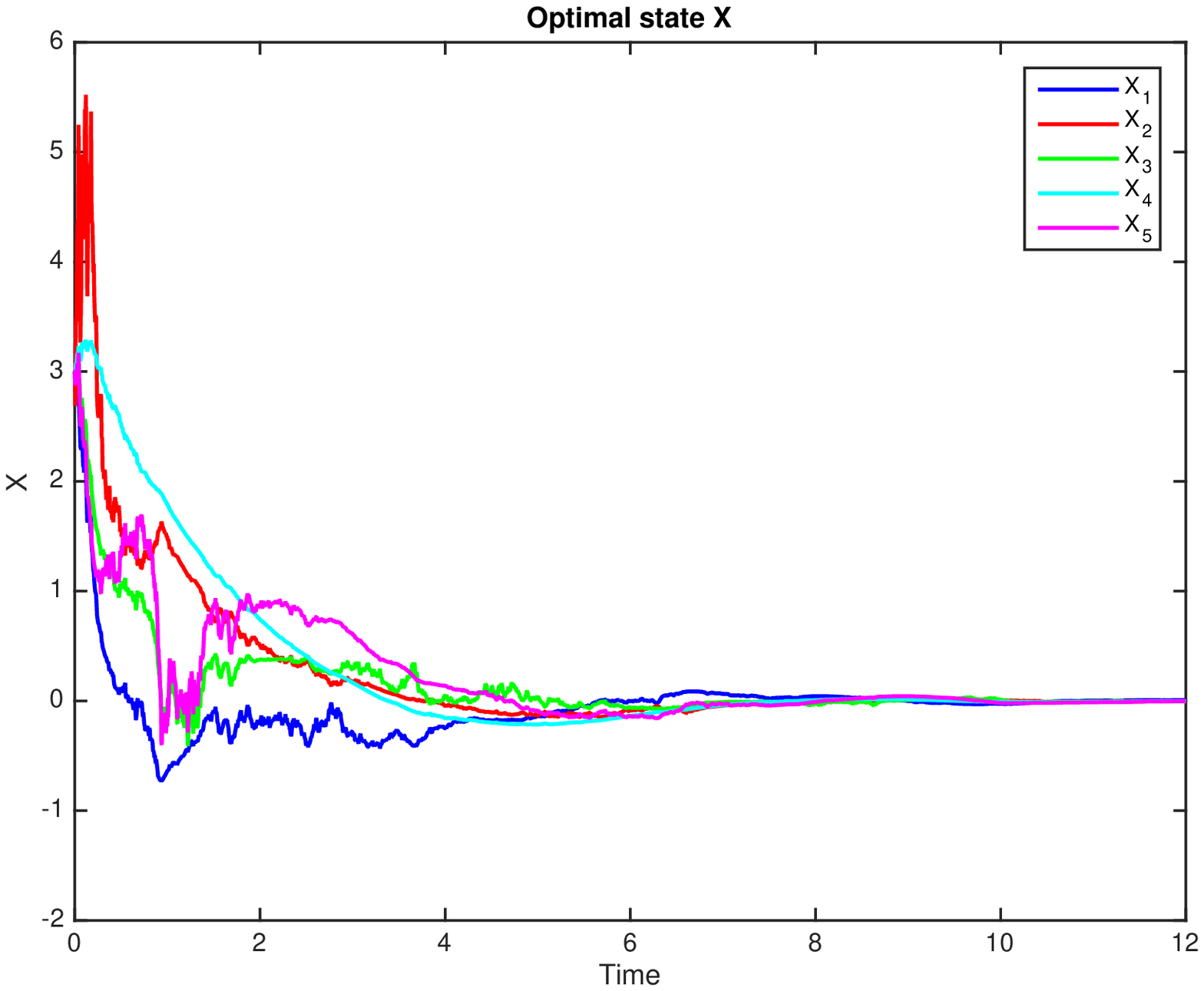} }\vspace{0.02in} 
\end{minipage}

\caption{Simulation results for solutions. (a): System state trajectory $X$ initialized by an arbitrary stabilizer $K^{(0)}$ in $2$-dimensional case; (b): Evolution of $P^*$ parameters in $2$-dimensional case; (c): The optimal state trajectory $X^*$ with optimal control $u^*=K^*X^*$ in $2$-dimensional case; (d): Evolution of $\widehat A$ parameters in $2$-dimensional case; (e): System state trajectory $X$ initialized by an arbitrary stabilizer $K^{(0)}$ in $5$-dimensional case; (f): The optimal state trajectory $X^*$ with optimal control $u^*=K^*X^*$ in $5$-dimensional case.}
\label{fig:subfig} 
\end{figure*}

Following \cite{Rami-Zhou-2000}, we consider an example with $n=5$, $m=2$, and $R=I$. To save space, we do not present the parameters of the problem, please see details in \cite{Rami-Zhou-2000}. Firstly, we initialize a 
stabilizer $$K^{(0)}=\begin{bmatrix} 
-0.2038 & 0.1082 & 0.3309 & -0.5508 & -0.2534 \\
0.3653 & 0.7350 & 0.8492 & -0.8452 & -0.5654
\end{bmatrix}.$$ 
Then, we read the data of state trajectory $X=(X_1, X_2, X_3, X_4, X_5)^\top$ with $K^{(0)}$, which is presented in Fig. \ref{fig:subfig} (e). 
By Algorithm \ref{algorithm}, we get 
$$P^*=\begin{bmatrix} 
0.3684 & 0.3093 & 0.2112 & 0.0272 & -0.3673 \\
0.3093 & 1.3393 & 1.2834 & -1.0028 & -0.7576 \\
0.2112 & 1.2834 & 1.6839 & -1.2101 & -0.7842 \\
0.0272 & -1.0028 & -1.2101 & 2.0571 & 0.0801 \\
-0.3673 & -0.7576 & -0.7842 & 0.0801 & 1.6467
\end{bmatrix}$$
after $15$ iterations in $7$ seconds with
$$\small{\mathcal R(P^*)=10^{-3}\cdot\begin{bmatrix} 
-0.2442 & 0.1120 & -0.0145 & 0.0413 & 0.0391 \\
0.1120 & 0.0245 & -0.1813 & 0.0381 & 0.0287 \\
-0.0145 & -0.1813 & 0.3078 & 0.0743 & -0.0869 \\
0.0413 & 0.0381 & 0.0743 & -0.1005 & 0.0014 \\
0.0391 & 0.0287 & -0.0869 & 0.0014 & 0.2771
\end{bmatrix},}$$
and $\|\mathcal R(P^*)\|=6.0846\times10^{-4}$. Comparatively, the accuracy of the corresponding result in \cite{Rami-Zhou-2000} is $\|\mathcal R(P)\|=2.6\times10^{-9}$. Because our RL method does not use any information of $A$ while the SDP method uses all the information of the system, the accuracy of the former is lower than the latter. The optimal control is $u^*=K^*X^*$, 
where
$$K^*=\begin{bmatrix} 
-0.3358 & -0.1425 & -0.0369 & 0.1088 & -0.0435 \\
0.4476 & 0.8551 & 0.9530 & -0.8889 & -0.3945
\end{bmatrix}.$$ 
The optimal state $X^*=(X^*_1,X^*_2, X^*_3, X^*_4, X^*_5)^\top$ is presented in Fig. \ref{fig:subfig} (f).

\ifCLASSOPTIONcaptionsoff
\newpage
\fi



%

\end{document}